\documentclass[10pt]{article}

\usepackage[utf8]{inputenc}
\usepackage{mathtools,
        mathrsfs,amssymb,amsthm, amsmath,
        bm}
\usepackage{comment}
\usepackage{hyperref}
\usepackage{url}
\usepackage[a4paper, left = 1in, right = 1in, top =1in, bottom = 1in]{geometry} 
\usepackage{dsfont}   
\usepackage[british]{babel}
\usepackage{enumitem}

\newtheorem*{thm}{Theorem}

\theoremstyle{definition}

\newtheorem*{rem}{Remarks}

\makeatletter
\renewenvironment{proof}[1][\proofname] {\par\pushQED{\qed}\normalfont\topsep6\p@\@plus6\p@\relax\trivlist\item[\hskip\labelsep\bfseries#1\@addpunct{.}]\ignorespaces}{\popQED\endtrivlist\@endpefalse}
\makeatother

\def\[#1\]{\begin{align*}#1\end{align*}}

\newcommand{\R}{\mathbb{R}}
\newcommand{\N}{\mathbb{N}}
\newcommand{\Z}{\mathbb{Z}}
\newcommand{\ceq}{\coloneqq}
\newcommand{\eps}{\varepsilon}

\renewcommand{\P}{\mathbb{P}}
\newcommand{\B}{\mathscr{B}}
\newcommand{\I}{\mathds{1}}

\begin{document}

\title{Sensitive Random Variables are Dense in Every $L^{p}(\R, \B_{\R}, \P)$}
\author{Yu-Lin Chou\thanks{Yu-Lin Chou, Institute of Statistics, National Tsing Hua University, Hsinchu 30013, Taiwan,  R.O.C.; Email: \protect\url{y.l.chou@gapp.nthu.edu.tw}.}}
\date{}
\maketitle

\begin{abstract}
We show that, for every $1 \leq p < +\infty$ and for every Borel
probability measure $\P$ over $\R$, every element of $L^{p}(\R, \B_{\R}, \P)$
is the $L^{p}$-limit of some sequence of bounded random variables 
that are Lebesgue-almost everywhere differentiable with derivatives
having norm greater than any pre-specified real number at every point
of differentiability. In general, this result provides, in some direction,
a finer description of an $L^{p}$-approximation for $L^{p}$ functions
on $\R$.\\

{\noindent \textbf{Keywords:}} differentiation; $L^{p}$-denseness; polygonal approximation; sensitive Borel random variables on the real line; statistics\\
{\noindent \textbf{MSC 2020:}} 60A10; 46E30; 26A24
\end{abstract}

\section{Introduction}

Regarding an $L^{p}$-approximation for $L^{p}$ functions with $1 \leq p < +\infty$,
it is well-known, apart from the classical result of the $L^{p}$-denseness
of simple measurable functions on an arbitrary measure space such that the co-zero-set of each of them has finite measure,
that compactly supported continuous functions on a Radon measure space whose
ambient space is locally compact Hausdorff are $L^{p}$-dense.

Given any probability measure $\P$ defined on the Borel sigma-algebra
$\B_{\R}$ of $\R$, we are interested in giving a more ``controlled”
$L^{p}$-approximation for every $1 \leq p < +\infty$, in terms of
the local steepness of the graphs of the approximating random variables.
For our purposes, a Borel random variable on ($\P$-probabilitized)
$\R$ is called $M$-\textit{sensitive} if and only if the random
variable is $\in L^{\infty}(\R, \B_{\R}, \P)$, is differentiable almost everywhere modulo Lebesgue measure,
and has the property that the normed derivative is $> M$ at every
point of differentiability. Here $M$ is a given real number $\geq 0$.
We stress that for a random variable on $\R$ to be $M$-sensitive
is a function-theoretic property of the random variable itself\footnote{In practice, the requirement of $\P$-essential boundedness may be strengthened to be boundedness; the definition thus does not really depend on the choice of the underlying probability measure $\P$.}; this
conceptual clarification would be advisable as the conventional terminology
in probability theory would lead the reader to instinctively take
our $M$-sensitiveness as a distributional property of a random variable
in analogy with ``absolutely continuous random variables”. Now
what we should like to prove is the possibility of $L^{p}$-approximating
any given element of $L^{p}(\R, \B_{\R}, \P)$ by $M$-sensitive random
variables.

We have communicated our intention mostly in a mathematically pure
language, which is for an \textit{a priori} concern of clarity\footnote{Using pure language can also be justified; probability theory is embedded in mathematics once we acknowledge that probability is viewed as a measure, and statistics, on the other hand, is embedded in mathematics once we interpret samples in terms of random variables. In this regard, studying properties of the Lebesgue spaces of random variables, being a fundamental workplace in probability theory (and statistics), advances the understanding of both fields.}.
But our finding also admits a natural, application-oriented interpretation.
As a statistic is in the broad sense a Borel function from $\R^{n}$
to $\R$, our result says that, given any $1 \leq p < +\infty$, every statistic of a single sample point drawn from a Borel distribution over $\R$, such as a Gaussian one, can be $L^{p}$-approximated by statistics with finite $p$-th moment that are ``as zigzagged as desired”, provided that the given statistic has finite $p$-th moment. 

\section{Proof}

We clarify, before the proof, some everyday terms present in the introductory
discussion that might not always admit a uniform usage in the related
literature. By a \textit{Borel probability measure over} $\R$ is
meant a probability measure defined on $\B_{\R}$; by a \textit{Borel random variable}
on $\R$, with $\R$ considered as a Borel probability space, we mean
an $\R$-valued, Borel-measurable function defined on $\R$. The set
of all reals $\geq 0$ will be denoted by $\R_{+}$.

The reader is invited to mind that our definition of the concept of
$M$-sensitiveness presumes Borel measurability; we should like to
prove

\begin{thm}\label{main}

If $1 \leq p < +\infty$, and if $\P$ is a Borel probability measure
over $\mathbb{R}$, then the $M$-sensitive random variables are $L^{p}$-dense
in $L^{p}(\R, \B_{\R}, \P)$ for every $M \in \R_{+}$. 

\end{thm}

\begin{proof}

For simplicity, we abbreviate $L^{p}(\R, \B_{\R}, \P)$ as $L^{p}(\P)$
in our proof. 

From the definition of a sensitive random variable, the class of all sensitive random variables is included in $L^{p}(\P)$. 

We claim that random variables of the form $\sum_{j=1}^{n}x_{j}\I_{V_{j}}$,
where $n \in \N$ with $x_{1}, \dots, x_{n} \in \R$ and $V_{1}, \dots, V_{n} \subset \R$
being disjoint open intervals, are $L^{p}$-dense in $L^{p}(\P)$.
Indeed, since the simple Borel random variables are $L^{p}$-dense
in $L^{p}(\P)$, by Minkowski inequality it suffices to prove that
every simple Borel random variable lies in the $L^{p}$-closure of the random variables of the desired form. In turn, it suffices to show that for
every $B \in \B_{\R}$ and every $\eps > 0$ there are some disjoint
open intervals $V_{1}, \dots, V_{N} \subset \R$ such that $|\I_{\cup_{j=1}^{N}V_{j}} - \I_{B}|_{L^{p}} < \eps$;
here $| \cdot |_{L^{p}}$ denotes the in-context $L^{p}$-norm. Every
Borel probability measure over a metric space is outer  regular (e.g. Theorem
1.1, Billingsley \cite{b}); so, given any $\eps > 0$ and any $B \in \B_{\R}$, there
is some open (with respect to the usual topology of $\R$, certainly)
subset $G$ of $\R$ such that $G \supset B$ and $\P(G \setminus B) < (\eps/2)^{p}$.
Since $G$ is also a countable union of disjoint open intervals $V_{1}, V_{2}, \dots$
of $\R$, and since $\P$ is a finite measure, there is some $N \in \N$
such that $\P(\cup_{j \geq N+1}V_{j}) < (\eps/2)^{p}$. But then Minkowski
inequality and the disjointness of these $V_{j}$ together imply 
\[
| \I_{\cup_{j=1}^{N}V_{j}} - \I_{B} |_{L^{p}} 
&= | \I_{\cup_{j}V_{j}} - \I_{\cup_{j \geq N+1}V_{j}} - \I_{B}|_{L^{p}}\\
&\leq | \I_{\cup_{j}V_{j} \setminus B}|_{L^{p}} + |\I_{\cup_{j \geq N+1}V_{j}}|_{L^{p}}\\
&< \eps,
\]
and the claim follows.

Let $X \in L^{p}(\P)$. If $\eps > 0$, choose some random variable
$\varphi_{0}$ that is a linear combination of finitely many indicators
of disjoint open intervals of $\R$ such that $|\varphi_{0} - X|_{L^{p}} < \eps/2$.
Then i) the random variable $\varphi_{0}$ is Borel and bounded, and ii) by construction, the set of all the points at which the random variable $\varphi_{0}$ is not differentiable is finite; in particular, the random variable $\varphi_{0}$ is differentiable almost everywhere modulo Lebesgue measure with $D\varphi_{0} = 0$ at every point of differentiability. 

Given any $M \in \R_{+}$,  let $b \ceq \lceil 2\eps^{-1}(M+1) \rceil,$
i.e. let $b$ be the least integer no less than the real number $2\eps^{-1}(M+1)$.
Define a function $\varphi_{1}: \R \to \R$ by assigning $0$ to each
$j/b$ with even $j \in \Z$, assigning $1$ to each $j/b$ with odd
$j \in \Z$, and taking the continuous linear interpolations between
all the points $j/b$ so that $|D\varphi_{1}| = b$ on $\R$ at every
point of differentiability and $\varphi_{1}$ is continuous. Then
$\varphi_{1}$ is Borel and bounded. We would like to remark
also that the set of all the points at which $\varphi_{1}$ is not
differentiable is countable, and each element of the set is isolated
(with respect to the standard topology of $\R$). 

If $Y \ceq \varphi_{0} + 2^{-1}\eps \varphi_{1}$, then $Y$ is Borel and bounded, and
\[
|Y - X|_{L^{p}} 
&\leq
|\varphi_{0} - X|_{L^{p}} + 2^{-1}\eps|\varphi_{1}|_{L^{p}}\\
&< \eps.
\]Further, the set of all the points at which $Y$ is not differentiable
is by construction countable with each element being isolated. In
particular, the random variable $Y$ is Lebesgue-almost everywhere
differentiable, and we have \[
|DY| 
&= 
|D\varphi_{0} + 2^{-1}\eps D\varphi_{1}|\\
&=
2^{-1}\eps|D\varphi_{1}|\\
&=
2^{-1}\eps b\\ 
&\geq 
M+1\\
&> M
\]at every point where $Y$ is differentiable. Since $Y$ is then $M$-sensitive,
the proof is complete. \end{proof}

We draw some posterior remarks herewith:

\begin{rem}

{ \ }

\begin{itemize}[leftmargin=*]

\item Our construction would be ``amicable” in the sense that
it does not depend on particularly deep results in analysis; and it
is inspired by the general proof idea of the marvelous result that
somewhere differentiable functions are meager in any given classical
Wiener space (e.g. Theorem 12.10, Krantz \cite{k}). 

\item In our setting, replacing the $L^{p}$-metric with the uniform
metric is not possible as the involved functions are not
necessarily even bounded.

\item Our proof applies to every Borel finite measure over $\R$;
the proof of Theorem 1.1. in Billingsley \cite{b} is applicable also
to Borel finite measures, and the value of corresponding $L^{p}$-norm
of $\varphi_{1}$ is immaterial.

\item Figuratively, the graph of a typical example of $Y$ constructed
in our proof would be a ``polygon” with finitely many ``missing parts”.\qed

\end{itemize}

\end{rem}


\begin{thebibliography}{1} 

\bibitem{b} 	
Billingsley, P. (1999). 	
\textit{Convergence of Probability Measures}, second edition.     
John Wiley.

\bibitem{k}
Krantz, S.G. (1991).
\textit{Real Analysis and Foundations}, first edition.
CRC Press.

\end{thebibliography}
\end{document}